\documentclass[11pt]{article}
\usepackage[papersize={21cm,29.7cm},left=2.5cm,right=2.5cm,top=3cm,bottom=2.5cm]{geometry}  
\usepackage{amsthm, amssymb, amsmath,amsfonts}
\usepackage[latin1]{inputenc}

\usepackage[pdftex]{graphicx}

\usepackage{color}

\numberwithin{equation}{section}

\newtheorem{thm}[subsection]{Theorem}
\newtheorem{prop}[subsection]{Proposition}
\newtheorem{lem}[subsection]{Lemma}
\newtheorem{cor}[subsection]{Corollary}

\theoremstyle{plain}
\newtheorem*{thm*}{Theorem}
\newtheorem*{prop*}{Proposition}
\newtheorem*{cor*}{Corollary}
\newtheorem*{conj*}{Conjecture}
\newtheorem*{question*}{Question}

\theoremstyle{definition}
\newtheorem{ex}[subsection]{Example}

\renewcommand{\emptyset}{\varnothing}

\renewcommand{\char}{\text{char}}
\renewcommand{\phi}{\varphi}

\newcommand{\Oe}{{\mathcal O}}

\newcommand{\Q}{{\mathbb Q}}

\newcommand{\Z}{{\mathbb Z}}

\newcommand{\VC}{{\mathrm {VC}}}

\newcommand{\KK}{{\widetilde K}}
\newcommand{\VV}{{\widetilde V}}
\newcommand{\LL}{{\widetilde L}}

\newcommand{\EE}{{\widetilde E}}
\newcommand{\FF}{{\widetilde F}}

\newcommand{\Hom}{{\mathrm{Hom}}}
\newcommand{\End}{{\mathrm{End}}}
\newcommand{\Aut}{{\mathrm{Aut}}}

\newcommand{\Gal}{{\mathrm{Gal}}}

\newcommand{\Tr}{\mathrm{Tr}}
\newcommand{\gp}{\mathfrak{p}}
\newcommand{\ga}{\mathfrak{a}}

\usepackage{fancybox}
\usepackage{fancyheadings}
\begin{document}
\addtolength{\parskip}{\medskipamount}
\thispagestyle{empty}
\title{The valuation criterion for normal basis generators}
\author{B. de Smit, M. Florence, L. Thomas
\thanks{This research was funded in part
by the European Commission under FP6 contract MRTN-CT-2006-035495. The third author was supported by a postdoctoral fellowship at the \' Ecole Polytechnique F\'ed\'erale de Lausanne in Switzerland.}}
\date{April 14, 2010}
\maketitle

\begin{quote}
{\bf Abstract: }
If $L/K$ is a finite Galois extension of local
fields, we say that the valuation criterion $VC(L/K)$ holds if there is an
integer $d$ such that every element $x \in L$ with valuation $d$ generates a
normal basis for $L/K$. Answering a question of Byott and Elder, we first prove
that $VC(L/K)$ holds if and only if the tamely ramified part of the extension
$L/K$ is trivial and every non-zero $K[G]$-submodule of $L$ contains a unit.
Moreover, the integer $d$ can take one value modulo $[L:K]$ only, namely
$-d_{L/K}-1$, where $d_{L/K}$ is the valuation of the different of $L/K$. 
When $K$ has positive characteristic, we thus recover a recent result of Elder
and Thomas, proving that $VC(L/K)$ is valid for all extensions $L/K$ in this
context. When $\char{\;K}=0$, we identify all abelian extensions $L/K$ for which
$VC(L/K)$ is true, using algebraic arguments. These extensions are determined
by the behaviour of their cyclic Kummer subextensions. 


\end{quote}

\section{Introduction}
Let $K$ be a field.
Let $L$ be a finite Galois extension of $K$, with Galois group~$G$.
An element $x\in L$ is called a \emph{normal basis generator} of $L$ over $K$,
or simply a \emph{normal} element of $L$ over $K$, if the conjugates of $x$
under $G$ form a basis of $L$ as a vector space over~$K$.
The normal basis theorem states that such an element exists.

Assume that $K$ (resp. $L$) is a local field, i.e., that it is complete with
respect to a discrete valuation $v_K\colon\; K^*\twoheadrightarrow \Z$ (resp.
$v_L\colon\; L^*\twoheadrightarrow \Z$). We consider the following question,
suggested by Byott and Elder~\cite{ByottElder07}.

\begin{question*}
Is there an element $d\in \Z$ so that every
$x\in L$ with $v_L(x)=d$ is normal over $K$?
\end{question*}

Note that, if such an integer $d$ exists, then all integers that are congruent
to $d$ modulo the ramification index $e_{L/K}$ satisfy the same property.

For example, when $K=\Q_2$ and $L=\Q_2(\sqrt{-1})$, then all
elements of odd valuation in $L$ are normal.
However, for $L=\Q_2(\sqrt{2})$,
the powers of $\sqrt 2$ give elements of $L$ of all possible valuations which
are not normal, so the answer to the question is no.

If $x\in L$ is normal over $K$, then the Galois conjugates are linearly
independent, so their sum, the trace $\Tr_{L/K}(x)$, is non-zero. It turns out
that it is quite easy to give a valuation criterion, formulated in the next
proposition, for this weaker property of having a non-zero trace.
We denote the valuation of the different of $L/K$ by $d_{L/K}$.

\begin{prop}\label{prop1}
Let $L/K$ be a finite separable extension of local fields, and let $d\in \Z$.
Every element of $L$ of valuation $d$ has non-zero trace
over $K$ if and only if the following two properties hold.
\begin{enumerate}
\item[$(1)$] $L/K$ is totally ramified.
\item[$(2)$] $d\equiv  - d_{L/K} - 1 \bmod [L:K]$.
\end{enumerate}
\end{prop}

The proof is given in Section 2.
In particular, Proposition \ref{prop1} implies
that the answer to the question is
positive if and only if the following statement holds:
\[
\VC(L/K): \quad
\textrm{ all }
x\in L^*
\textrm{ with }
v_L(x) \equiv -d_{L/K} -1  \bmod  \, e_{L/K}
\textrm{ are normal over }
K.
\]
We will call this statement the \emph{valuation criterion} for normal basis
generators of $L$ over~$K$.

\vskip6mm

\begin{prop} \label{VCP}
Let $L/K$ be a finite Galois extension of local fields.
The valuation criterion $\VC(L/K)$ holds if and only if
the following two conditions hold.
\begin{enumerate}
\item[$(1)$]
$L/K$ is totally ramified and $[L:K]$ is a power of the residue characteristic $p$.
\item[$(2)$]
Every non-zero $K[G]$-submodule of $L$ contains an element of valuation~$0$.
\end{enumerate}
\end{prop}

The proof will be given in
Section 3. It is based on a duality result for the set of
valuations of elements of a sub-$K$-vector space of~$L$.

Note that, if the residue field of $K$ has characteristic zero, then Proposition \ref{VCP} can be restated as:  $\VC(L/K)$ holds if and only if $L=K$.  

Note also that, in condition $(2)$, we may restrict to the minimal
non-zero $K[G]$-submodules of $L$. The condition in (1) that $[L:K]$ is a power of $p$ can be omitted --- we will see in the proof that it is implied by condition (2). The condition that $L/K$ should be totally ramified can not be omitted.

If $K$ has characteristic $p>0$ and condition $(1)$ in Proposition \ref{VCP} holds,
then condition $(2)$ also holds
because every nonzero $G$-stable $K$-vector subspace of $L$
then contains $K$; see \cite[Ch.~IX, Th.~2]{SerCL}.
In the equal characteristic case Proposition \ref{VCP} therefore
tells us that condition $(1)$ implies $\VC(L/K)$, which
was shown already by Thomas~\cite{Lara08} and by Elder~\cite{Eld09}.

From now on, we assume that $K$ is of mixed characteristic $(0,p)$.

In general, condition $(1)$ of Proposition \ref{VCP} is not
sufficient for $\VC(L/K)$ to hold.  For example, consider
the extension $\Q_2(\sqrt 2))/\Q_2$.
More generally, elements in extensions of $K$ that are strictly
contained in $L$ are never normal elements
of $L$, so one sees that the condition $p\nmid -d_{L/K}-1$ is necessary
for $\VC(L/K)$ to hold.  For cyclic extensions of degree $p$ this
condition is also sufficient; cf. \cite{ByottElder07}.
However, we will see in Example~\ref{cp2} that this condition is not sufficient
for cyclic extensions of degree~$p^2$.

By condition $(2)$ in Proposition \ref{VCP}, we can easily identify the Kummer
extensions for which the valuation criterion holds.  Recall that $L/K$ is a
\emph{Kummer extension} if there is a number $m$ so that $K$ contains a
primitive root of unity of order $m$, and $\Gal(L/K)$ is abelian of
exponent~$m$.  Then the characteristic of $K$ does not divide $m$, and by
Kummer theory we have $L=K(\root m \of W)$ for $W={L^*}^m\cap K^*$.  If, in
addition, we have $v_K(W)\subset m\Z$, then $L$ is obtained by adjoining $m$-th
roots of units of the valuation ring of $K$, and we say that $L$ is a
\emph{unit root Kummer extension} of~$K$.  For example, $\Q_2(\sqrt{-1})$ is a
unit root Kummer extension of $\Q_2$, whereas $\Q_2(\sqrt{2})$ is not.

\begin{thm}\label{thm1}
Let $L/K$ be a totally ramified Kummer extension of local fields
whose degree is a power of the residue characteristic~$p$.
Then $\VC(L/K)$ holds if and only if $L$ is a unit root Kummer extension
of~$K$.
\end{thm}

In Section~\ref{ProofThm2} we give the proof, and we show 
how the general abelian case can be reduced to the Kummer case. 
Precisely, we will show the following theorem.

\begin{thm}\label{thm2}
Let $L/K$ be a totally ramified abelian extension of local fields
whose degree is a power of the residue characteristic~$p>0$.
Let $m$ be the exponent of $\Gal(L/K)$, and let $r\mid m$ be the
number of $m$-th roots of unity inside~$K$.
If $p=2$ and $8\mid m$, assume that $r\ne 2$.
Then $\VC(L/K)$ holds
if and only if every cyclic subextension $F/E$ of $L/K$ of degree $r$ is a unit root Kummer extension.
\end{thm}

If $r=1$ in Theorem \ref{thm2}, then the condition in the theorem is trivially
satisfied, so that the valuation criterion holds.  In particular, if $K$ does
not contain a primitive $p\/$th root of unity, then the valuation criterion
holds for every abelian $p$-extension $L$ of $K$.

When $p=2$, the additional hypothesis is
due to the fact that $(\Z/2^k\Z)^*$ is not cyclic
for $k\ge 3$.
If $K=\Q_2(\sqrt{-2})$ and $L=\Q_2(\mu_{32})$, then we have $m=8$ and $p=r=2$.
Theorem~\ref{thm2} implies that $\VC(F/E)$ holds for all extensions $E\subset
F$ with $K\subset E \subset F\subset L$ of degree at most 4.
We will see in Example \ref{z32} that
$\VC(L/K)$ does not hold. Thus,
we cannot omit the condition when $p=2$ in Theorem~\ref{thm2}.

\section{The valuation criterion for having non-zero trace}

The purpose of this section is to prove Proposition~\ref{prop1}.

As before, $K$ denotes a local field and $v_K\colon\; K^*\twoheadrightarrow \Z$
is the normalized valuation.  Inside $K$ we consider the valuation ring
$\Oe_K=\{x\in K^*\colon\; v_K(x)\ge 0\} \cup \{0\}$, its maximal ideal
$\gp_K$ and its unit group $\Oe_K^*=\{x\in K^*\colon\; v_K(x)=0\}$.
The valuation $v_L(\ga)$ of a fractional ideal $\ga$ is the valuation
of any of its generators, so $v_K(\gp_K^i)=i$ for all $i\in \Z$.

Suppose now that $L$ is a finite separable field extension of~$K$.  Then $L$
has the structure of a local field as well. We denote by $\Tr_{L/K}$ the trace
map from $L$ to~$K$. Two integers are naturally attached to the extension
$L/K$. The first one is its ramification index $e_{L/K}$, given by the equality
$v_L(K^*)=e_{L/K}\Z$.
The second one is $d_{L/K}$, the valuation of the
different of $L$ over $K$, which is characterized by the property
that
$$ i \ge -d_{L/K} \iff \Tr_{L/K}(\gp_L^i)\subset \Oe_K $$
for all $i\in\Z$; cf.  \cite[Ch.~III]{SerCL}.
Using this it is easy to identify the traces of ideals:
for every $i\in \Z$ we have
\begin{equation}\label{different}
\Tr_{L/K}(\gp_L^{-d_{L/K}+i})=\gp_K^{\left\lfloor \frac{i}{e_{L/K}}\right\rfloor},
\end{equation}
where $\lfloor x \rfloor$ denotes the largest integer $n$ with $n\le x$.

\begin{proof}[Proof of Proposition \ref{prop1}]
For any $d\in \Z$ let us consider the map
$$
\phi\colon\; \gp_L^{d} /\gp_L^{d+1} \longrightarrow
\Tr_{L/K}(\gp_L^d)/ \Tr_{L/K}(\gp_L^{d+1})
.
$$
induced by the trace map.
Denoting the residue field of $K$ by $k$, and the degree of the residue field
extension of $L/K$ by $f=[L:K]/e_{L/K}$, we see that the domain of
$\phi$ is an $f$-dimensional vector space over $k$. By \eqref{different} above,
the codomain of $f$ is a vector space over $k$ which
is of dimension $1$ if $d \equiv -d_{L/K}-1 \bmod [L:K]$ and of dimension $0$
if $d \not\equiv -d_{L/K}-1 \bmod [L:K]$.
Since $\phi$ is a $k$-linear surjective map it follows that $\phi$ is an
isomorphism if and only if both conditions (1) and (2) in the proposition are satisfied.

We now distinguish two cases. If $\phi$ is an isomorphism, then for any
element $x\in L^*$ with $v_L(x)=d$ we have $\Tr_{L/K}(x)\not\in
\Tr_{L/K}(\gp_L^{d+1})$, which implies that $\Tr_{L/K}(x)\ne 0$.

If, on the other hand, $\phi$ is not an isomorphism, then we can choose $x\in
\gp_L^{d}$ so that $(x \bmod \gp_L^{d+1} )$ is a non-zero element of the kernel
of $\phi$.
We then have $\Tr_{L/K}(x)\in \Tr_{L/K}(\gp_L^{d+1})$, so that
$\Tr_{L/K}(x)=\Tr_{L/K}(y)$ for some $y\in \gp_L^{d+1}$. But this implies that
$x-y$ is an element of $L$ of valuation $d$ and trace $0$.
This completes the proof of Proposition \ref{prop1}.
\end{proof}

\section{The set of valuations of elements in a linear subspace}

In this section we prove Propostion \ref{VCP}. The key tools 
we develop for this, and for applications in the next section,
are basic properties of the set of valuations of elements
in subspaces of a field extension.

Let $L/K$ be a finite separable
totally ramified extension of local fields of degree $n$. 
Let $v\colon\;L^* \rightarrow \Z/n\Z$ be the map given by $x\mapsto v_L(x) \bmod n.$
For any sub-$K$-vector space $V$ of $L$ we define the set $s(V)$ by
$$
s(V)=\{v(x)\colon\; x \in V,\ x\ne 0\} \subset \Z/n\Z.
$$

Since $L$ is totally ramified over $K$, Proposition~\ref{prop1} implies
that exactly one residue class modulo $e_{L/K}$ does not occur as the valuation
of an element of the ``trace zero'' hyperplane. This is a general fact that
holds for all sub-$K$-vector spaces of~$L$.

\begin{lem} \label{NiceLemma}
For every sub-$K$-vector space $V$ of~$L$, we have $\# s(V)=\dim_K(V)$.
\end{lem}

\begin{proof}
Let $S$ be a subset of $V$ such that $v$ maps $S$ bijectively to $s(V)$.
We will show that $S$ is a basis of $V$ over $K$.

Note that in a non-trivial $K$-linear combination of elements of $S$,  all
non-zero terms have valuations which are distinct modulo~$n$. Thus, these
valuations are distinct and their minimum is the valuation of the sum.
In particular, this sum is not zero in $L$,
and it follows that $S$ is a linearly independent set over $K$.

Now let $W$ be the sub-$K$-vector space of $V$ generated by~$S$.
Consider the finitely generated $\Oe_K$-submodules
$W^0=W\cap \Oe_L\subset V^0=V\cap \Oe_L$ of $\Oe_L$.
Using the fact that $v(W^0\backslash\{0\})=
v(V^0\backslash\{0\})$ one sees that $V^0=W^0+\gp_KV^0$.
Nakayama's lemma then implies that $V^0=W^0$.
It follows that $V=KV^0=KW^0=W$.
\end{proof}

Recall that we have a non-degenerate symmetric
$K$-bilinear form on $L$ given by $(x,y) \mapsto \Tr_{L/K}(xy)$.
For any sub-$K$-vector space $V$ of $L$, the orthogonal space
$V^\perp\subset L$ is isomorphic to the $K$-dual of $L/V$.

\begin{lem} \label{duality}
Let $\overline d=(-d_{L/K}-1\bmod n\Z)\in \Z/n\Z$.
For every sub-$K$-vector space $V$ of $L$ the set
$s(V^\perp)$ is the complement in $(\Z/n\Z)$ of the set $\overline d - s(V)$.
\end{lem}
\begin{proof}
For non-zero $x\in V$ and $y\in V^\perp$ we have $\Tr_{L/K}(xy)=0$,
so that $v(x)+v(y)=v(xy)\ne \overline d$ by Proposition \ref{prop1}.
It follows that $\overline d\not \in s(V)+s(V^\perp)$, so that
$s(V^\perp)$ is contained in the complement in $(\Z/n\Z)$ of the set $\overline d - s(V)$.
One sees with Lemma \ref{NiceLemma} that these two sets have the same cardinality,
the codimension over $K$ of $V$ in $L$, so they are equal.
\end{proof}

For example, taking $V=K$, the orthogonal space $V^\perp$ is the kernel of the trace,
and $s(V^\perp)=(\Z/n\Z)\backslash \{\overline d\}$.

\begin{proof}[Proof of Proposition~\ref{VCP}]
If $L/K$ is not totally ramified, then 
Proposition \ref{prop1} implies that $\VC(L/K)$ is false;
 we may thus assume that $L/K$ is totally ramified of degree $n$.

Clearly, $\VC(L/K)$ holds if and only if no $K[G]$-submodule $V$
strictly contained in $L$ contains an element $x$ with $v(x)=
\overline d$, where $\overline d=(-d_{L/K}-1\bmod n\Z)\in \Z/n\Z$.
This means that $\overline d\not \in s(V)$ for all such~$V$.
By duality, the map $V \mapsto W=V^\perp$ gives a bijection from the set
of $K[G]$-submodules $V$ of $L$ that are strictly contained in $L$
to the set of non-zero $K[G]$-submodules $W$ of~$L$.
By Lemma \ref{duality}, we have $\overline d\not \in s(V) \iff 0 \in s(V^\perp)$,
so we deduce that $\VC(L/K)$ holds if and only if
$0\in s(W)$ for every non zero $K[G]$-submodule
$W$ of $L$. Thus we see that $\VC(L/K)$ is equivalent to condition $(2)$.

It remains to show that condition $(2)$ implies that $[L:K]$ is a power of the
residue characteristic $p$. To see this, let $L'/K$ be the maximal tamely
ramified subextension of $L$ over $K$.
Then condition $(2)$ also holds for $L'/K$. So, by what we proved
already, $\VC(L'/K)$ holds. Since $L'/K$ is tamely ramified, we have $d_{L'/K}=e_{L'/K}-1$ \cite[Ch.~III,~\S6,~Prop.~13]{SerCL}.
Then $\VC(L'/K)$ implies that non-zero elements of $K$ are normal basis
generators for $L'$ over $K$, so $L'=K$.
\end{proof}

\begin{ex}\label{z32}
Let $K=\Q_2(\sqrt{-2})$ and for $n$ a power of $2$ let
$\zeta_n$ denote a root of unity of order $n$ in a fixed algebraic closure of $K$.
For $d=2,4,8, 16, \ldots$ the field $L_d=\Q_2(\zeta_{4d})$ is cyclic of order $d$
over $K$, and its Galois group $G_d$ is generated by the automorphism
$\sigma\colon\;\zeta_{4d}\mapsto\zeta_{4d}^3$.

To check whether condition $(2)$ of Proposition \ref{VCP} holds for the
extension $L_4/K$, note that the minimal non-zero $K[G_4]$-submodules
of $L_4$ are the kernels of the elements $f(\sigma)$ acting on $L_4$,
where $f$ ranges over the irreducible factors
$X-1$, $X+1$ and $X^2+1$ of
$X^4-1\in K[X]$.
Thus, 
the minimal $K[G_4]$-submodules of $L_4$ are $K$ and $\ker\Tr_{L_2/K}$, and $\ker\Tr_{L_4/L_2}$.
These contain the units: $1, \zeta_{4}$ and $\zeta_{16}$. So, by
Proposition \ref{VCP}, $\VC(L_4/K)$ holds.

Let us try the same for $L_8/K$.  The polynomial $X^8-1$ factors over $K$
into irreducible polynomials as follows:
$$
X^8-1=(X-1)(X+1)(X^2+1) (X^2+\sqrt{-2} X -1) (X^2-\sqrt{-2} X +1),
$$
so in addition to the minimal submodules we found inside $L_4$, which we know contain units,
we need to consider two sub-$K[G_8]$-modules of $K$-dimension $2$ inside $\ker \Tr_{L_8/L_4}$.

Now $\zeta_{32}$ is contained in $\ker_{L_8/L_4}$, so let us put
$x=\sigma^2(\zeta_{32})-\sqrt{-2}\sigma(\zeta_{32})-\zeta_{32}$.
Then the sub-$K$-vector space $V$ of $L$ generated by $x$ and $\sigma(x)$ is
a minimal $K[G]$-submodule of $L$.
One now checks with a computation that $v_{L_8}(x)=10$ and $v_{L_8}(\sigma(x)-x)=14$, so that
$\{2,6\}\subset s(V) \subset (\Z/8\Z)$.
With Lemma \ref{NiceLemma} we see that $s(V)=\{2,6\}$, and
it follows from Proposition \ref{VCP} that $\VC(L_8/K)$ does not hold.
From the proof above we see that an element of $L$ of valuation $-d_{L_8/K}-1$ can be found inside
the $K[G_8]$-submodule $V^\perp$ of $L_8$.
\end{ex}

We conclude this section with some easy consequences of Proposition \ref{VCP}.

\begin{cor}\label{subcor}
If $K\subset L \subset M$ are finite extensions of local fields with $M$ and $L$
both Galois over $K$, then
$\VC(M/K)$ implies $\VC(L/K)$.
\end{cor}

Note that, in the setting of this corollary, a normal element for $M$ over $K$ is not
necessarily normal over $L$; see \cite{BlessJohn86} for an easy example,
and~\cite{Faith57, Hachen97}.  We do not know whether the implication
$\VC(M/K)\implies \VC(M/L)$ always holds, even for abelian extensions.  
However, it does hold in a particular setting of Kummer extensions --- see Lemma~\ref{topkummer}.

\begin{lem}\label{tamebasechange}
Let $L/K$ be a totally ramified Galois extension of local fields
whose degree $n$ is a power of the residue characteristic $p>0$.
For every finite tamely ramified extension $\KK/K$ we have
$$
\VC(\KK L/\KK) \implies \VC(L/K).
$$
\end{lem}
\begin{proof}
Put $\LL=\tilde K L$. Note first that $[\LL: \KK]=[L:K]$, because the tame part of $L/K$ is trivial.
Let $V$ be a $K$-submodule of $L$. Put $\VV= \KK V$; it is a $\KK$-submodule of $\LL$. Note that we have the obvious inclusion $$e_{\KK/K} s(V) \subset s(\VV).$$ 
Since $e_{\KK/K}$ is coprime to $p$ and therefore to $n$, and since $s(V)$ and $s(\VV)$ are both of cardinality $\mathrm{dim}_K(V)$ (Lemma \ref{NiceLemma}), it follows that this inclusion is in fact an equality.  Thus, $V$ contains a unit of $L$ (i.e., $0$ belongs to $s(V)$) if and only if
$\VV$ contains a unit of $\LL$ (i.e., $0$ belongs to $s(\VV)$). Assume now that $\VC(\LL /\KK)$ holds. Then Proposition \ref{VCP} implies that
 the space $\VV$ contains a unit, so that $V$ contains a unit too. The $K$-submodule $V$ of $L$ being arbitrary,
again by Proposition \ref{VCP} it follows that $\VC(L/K)$ holds.
\end{proof}

\section{Applications to abelian extensions}\label{ProofThm2}

In this section we consider only abelian extensions $L/K$. 
We will show that Theorems \ref{thm1} and \ref{thm2} hold
by using Proposition \ref{VCP}.

\begin{proof}[Proof of Theorem~\ref{thm1}]
Suppose that $G=\Gal(L/K)$ is of exponent $m$, and that $K$ contains a primitive $m\/$th root of unity. Put $n=\#G$.
By Kummer theory there exists a $K$-basis $R$ of $L$ such that
$\{r^m \colon\; r\in R\}$ is a full set of coset representatives of $K^*\cap
L^{*m}$ modulo $K^{*m}$.
Thus, $L/K$ is a unit root Kummer extension if and only if $m \mid v_K(r^m)$ for all $r\in R$.

The $K$-algebra $K[G]$ is totally split, and $L$ is free of rank $1$ over
$K[G]$, so $L$ is the direct sum of its $n$ distinct minimal non-zero
$K[G]$-submodules, and they all have dimension $1$ over $K$. These
submodules are therefore the modules $Kr$ with $r\in R$.

By Proposition \ref{VCP}, we know that
$\VC(L/K)$ holds if and only if all minimal non-zero $K[G]$-submodules of $L$
contain an element of $\Oe_L^*$. The result now follows by noting that
$$Kr\cap \Oe_L^* \ne\emptyset \iff n\mid v_L(r) \iff nm \mid v_L(r^m) \iff m \mid v_K(r^m).$$
\end{proof}

\begin{ex}\label{cp2}
Suppose that $K$ contains $\mu_{p^2}$, that $u\in \Oe_K^*$ is not a $p$-th power
and that $\pi\in K^*$ satisfies $v_K(\pi)=1$. Then $u\pi^p$ is
not a $p$-th power in $K$, so
$L=K(\root p^2 \of {u\pi^p})$ is a cyclic extension of degree~$p^2$.
It is a Kummer extension, and by Theorem~\ref{thm2} it does not
satisfy $\VC(L/K)$. However, the intermediate field
$M=K(\root p \of u )$ satisfies both $\VC(M/K)$ and $\VC(L/M)$.
Note that $-1-d_{L/K}\equiv -1-d_{L/M} \not\equiv 0 \bmod p$.
\end{ex}

In order to prove Theorem \ref{thm2} we first present two auxilliary results.

\begin{lem}\label{cyclic}
If $L/K$ is an abelian extension, and $K$ has characteristic $0$, then
$\VC(L/K)$ holds if and only if $\VC(E/K)$ holds for all
intermediate fields $K\subset E \subset L$ for which $E$ is
cyclic over~$K$.
\end{lem}

\begin{proof}
Let $V$ be a minimal non-zero $K[G]$-submodule of $L$. Since the group ring $K[G]$ is a product of fields, the image $F$ of $K[G]$ in $\End_K(V)$ is a field.
Let $H$ be the kernel of the canonical map $G\to F^*$, and let $E=L^H$. Then  $V \subset E$, and $G/H=\mathrm{Gal}(E/K)$ is cyclic because it embeds into $F^*$. We have just shown that every minimal non-zero $K[G]$-submodule of $L$
is contained inside a field $E$ with $K\subset E\subset L$ and $E/K$ cyclic.
The lemma now follows from Proposition \ref{VCP}.
\end{proof}

\begin{lem}\label{topkummer}
Let $M/K$ be a Galois extension of local fields
and let $L$ be a subfield of $M$ which is normal over~$K$.
If $M/L$ is abelian of exponent $r$ and $K$ contains a
root of unity of order $r$ then
$$
\VC(M/K) \implies \VC(M/L).
$$
\end{lem}

\begin{proof}
Suppose that $x\in M$ is normal over $K$. We will show that $x$ is also
normal over $L$.  We write $G=\Gal(M/K)$ and $H=\Gal(M/L)$. Then we
know that $x$ has trivial annihilator in the ring $K[G]$, so it also has
trivial annihilator in $K[H]$.  The latter is a totally split $K$-algebra.
Since $L[H]$ is a totally split $L$-algebra with the same number of components,
each of its nonzero ideals contains a nonzero element of $K[H]$. Thus, $x$ also
has a trivial annihilator in $L[H]$, and $x$ is normal over~$L$.

If we assume that $\VC(M/K)$ holds, then for some $d\in \Z$ all $x\in M^*$ of
valuation $d$ are normal over $K$. We just showed that all these $x$ are
then normal for $M/L$ too. With Proposition \ref{prop1} this implies
that $d\equiv -d_{M/L}-1\bmod [M:L]$, and it follows that $\VC(M/L)$ holds.
\end{proof}

The core of our argument lies in the proof of the following two lemmas.

\begin{lem}\label{hard1}
Let $L/K$ be a totally ramified abelian extension of local fields of mixed
characteristic $(0, p)$, whose degree is a power of $p$.
Let $\mu_p$ be the group of $p\/$th roots of unity
in an algebraic closure of $K$. 
If $\mu_p\not\subset K$ and
$L(\mu_p)/K(\mu_p)$ is a Kummer extension then $\VC(L(\mu_p)/K(\mu_p))$ holds.
\end{lem}

\begin{proof}
Put $\KK=K(\mu_p)$, $\LL=L(\mu_p)$ and $G=\Gal(L/K)=\Gal(\LL/ \KK)$. Let $m$ be the exponent of $G$.
Note first that $\KK[G]$ is a totally split
$\KK$-algebra, so $\LL$ is the direct sum of its minimal non-zero submodules,
which are exactly the eigenspaces $$E_\chi=\{x\in L\colon\; g x= \chi(g)x \text{ for all }g\in G\},$$
where $\chi$ ranges over $\Hom(G,\KK^*)$. 
We will consider the sets $s(E_\chi)$, where $s$ is as in Section 3 for the extension $\LL/\KK$.
Since $E_\chi$ is $1$-dimensional over $\KK$, these are one element sets.

By our assumptions, there is an element $\sigma\in\Gal(\LL/K)$ that acts
non-trivially on $\mu_p$, so it acts on $\mu_m\subset \KK^*$
by raising elements to the power $c$ for some $c\in \Z$, which is not $1$ modulo $p$.

Now on the one hand $s(\sigma(E_\chi))=s(E_\chi)$, because $\sigma$ preserves the valuation.
On the other hand, we have $x^c \in E_{\chi^c}=\sigma(E_\chi)$ for all $x\in E_\chi$,
so $s(\sigma(E_\chi))\supset c s(E_\chi)$. Since $c-1$ is coprime to $[\LL:\KK]$ 
and $ s(E_\chi)$ is a set consisting of a single element, this implies that $s(E_\chi)=\{0\}$.
By Proposition \ref{VCP}, it then follows that $\VC(\LL/\KK)$ holds.
\end{proof}

\begin{lem}\label{hard2}
Let $L/K$ be a totally ramified cyclic extension of local fields of
mixed characteristic $(0,p)$ whose degree $n$ is a power of $p$.
Let $r\mid n$ be the number of $n$-th roots of unity in $K^*$.
Assume that $p\mid r$, and if $p=2$ and $8\mid n$ assume that $r\ne 2$.
For the chain of fields $K\subset L_r\subset L_p\subset L$ where $[L:L_p]=p$ and $[L:L_r]=r$
we then have
$$\VC(L/L_r) \mathrm{\ and \ } \VC(L_p/K) \implies \VC(L/K).$$
\end{lem}

\begin{proof}
\medskip
If $\sigma$ is a generator of $G=\Gal(L/K)$,
then the minimal non-zero $K[G]$-submodules of $L$ are the spaces
$V_f=\{x\in L\colon\; f(\sigma)\cdot x=0\}$, where $f$ ranges over the
monic irreducible factors of $X^n-1$ in $K[X]$.  Let
$\mu_n$ be the group of $n$-th roots of unity in some algebraic closure of $K$.
Every $z\in\mu_n$ either has order less than $n$, so that it is a zero of
$X^{n/p}-1$, or it has order $n$ and then $z^{n/r}$ is a root of
unity of order $r$ in $K$. Thus we see that
$$
X^n-1=(X^{n/p}-1) \; \prod_{
    \begin{array}{c}
    \zeta \in K^* \\
    \#\langle \zeta\rangle = r
    \end{array}
    }  (X^{n/r}-\zeta).
$$
We claim that the polynomials $X^{n/r}-\zeta$ are all irreducible in $K[X]$.
In order to see this, note first that $\Gal(K(\mu_n)/K)\subset \Aut(\mu_n)=(\Z/n\Z)^*$. Let $H$ be the kernel of the map $(\Z/n\Z)^*\to (\Z/r\Z)^*$. Then $\Gal(K(\mu_n)/K)$ is a subgroup of $H$, not contained in the kernel of the map $(\Z/n\Z)^*\to
(\Z/pr\Z)^*$.  Under the assumption we made, $H$ is the only such group, so
that $\Gal(K(\mu_n)/K)=H$ and $K(\mu_n)$ has degree $n/r=\#H$ over~$K$. This shows the claim.

Suppose now that $f$ is an irreducible factor of $X^n-1$. If $f$ is a factor of
$X^{n/p}-1$ then $V_f$ is a $K[G]$-submodule of $L_p$. Otherwise, $f=X^{n/r}-\zeta$
for some $\zeta\in K^*$ of order $r$, and $V_f$ is an eigenspace for the action
of $G_r=\langle \sigma^{n/r}\rangle=\Gal(L/L_r)$ on $L$, so it is an
$L_r[G_r]$-submodule of $L$.
Thus, every minimal $K[G]$-submodule of $L$ is a $K[G]$-submodule of $L_p$
or it is an $L_r[G_r]$-submodule of $L$. Our statement now follows with
Proposition \ref{VCP}.
\end{proof}

We can now prove our main theorem.

\begin{proof}[Proof of Theorem~\ref{thm2}]
Let us first assume that $\VC(L/K)$ holds, and suppose that we have intermediate fields
$K\subset E\subset F\subset L$ with $F/E$ cyclic of degree $r$.
Then $\VC(F/K)$ holds by Corollary \ref{subcor} and $\VC(F/E)$ holds by
Lemma \ref{topkummer}. With Theorem \ref{thm1}, we then see that $F/E$ is a
unit root Kummer extension.

To show the other implication, we assume that
$F/E$ is a unit root Kummer extension for all $E$, $F$ as above,
which by Theorem \ref{thm1} implies that $\VC(F/E)$ holds too.
We also assume that we are not in the case where $p=r=2$ and $8\mid m$.
We will prove that $\VC(L/K)$ holds, and by Lemma \ref{cyclic} it suffices to do this
under the additional hypothesis that $L/K$ is cyclic. So we assume that $L/K$
is cyclic of degree $n$.  We now consider two cases: $r=1$ and $r>1$.

If $r>1$ then we proceed with induction on $n/r$, where $n=[L:K]$.  If $n/r=1$,
then we may take $E/F$ to be $L/K$, and we are done.  If $n/r > 1$, then
consider $K\subset L \subset L_r\subset L_p\subset L$ as in Lemma \ref{hard2}.
By the induction hypothesis we then see that $\VC(L_p/K)$ holds. Taking $F/E$
to be $L/L_r$ we see that $\VC(L/L_r)$ holds.  Thus, Lemma \ref{hard2}
completes the proof in the case that $r>1$.

Now suppose that $r=1$. Put $\KK=K(\mu_p)$ and $\LL=L(\mu_p)$.
We claim that $\VC(\FF/\EE)$ holds whenever
$\KK\subset \EE \subset \FF \subset \LL$ and $\FF/\EE$ is a Kummer extension.
By Galois theory, $\EE$ and $\FF$ are of the form
$\EE=E(\mu_p)$ and $\FF=F(\mu_p)$ for
certain intermediate fields $K \subset E \subset F \subset L$.
We then have $\mu_p\not\subset E$, because $L/K$ has only a trivial tame part,
so by Lemma \ref{hard1} we see that
$\VC(\EE/\FF)$ holds, as claimed.
By using the already proven case $r>1$ of Theorem~\ref{thm2},
it follows that $\VC(\LL/\KK)$ holds. By Lemma \ref{tamebasechange} this implies that
$\VC(L/K)$ holds.
\end{proof}


\begin{thebibliography}{99}
\bibitem{BlessJohn86} \textsc{D. Blessenohl, K. Johnsen}, {\em Eine Versch\"{a}rfung des Satzes von der Normalbasis}, Journal of Algebra, {\bf 103} (1986), 141-159.
\bibitem{ByottElder07} \textsc{N. P. Byott, G.G. Elder}, {\em A valuation criterion for normal bases in elementary abelian extensions}, Bull. London Math. Soc., {\bf 39} (5) (2007), 705-708.
\bibitem{Eld09} \textsc{G.G. Elder}, {\em A valuation criterion for normal basis generators in local fields of characteristic p}, Arch. Math., {\bf 94} (2010), 43-47.
\bibitem{Faith57}  \textsc{C.C. Faith}, {\em Extensions of normal bases and completely basic fields}, Transactions of the American Mathematical Society, {\bf 85} (1957), 406 - 427.
\bibitem{Hachen97} \textsc{D. Hachenberger}, {\em Finite fields, Normal bases and completely free elements}, Kluwer Academic Publishers, 1997.
\bibitem{Lara08} \textsc{L.Thomas}, {\em A valuation criterion for normal basis generators in equal positive characteristics}, Journal of Algebra, preprint 2008.
\bibitem{SerCL} \textsc{J.-P. Serre}, {\em Corps locaux}, fourth edition, Hermann, Paris, 1968.
\end{thebibliography}
\end{document}